\documentclass[11pt,leqno,a4paper]{article}

\makeatletter
\usepackage[reqno]{amsmath}
\usepackage{epsfig,changebar}
\usepackage{graphicx}
\usepackage{shadow}
\usepackage{indentfirst}
\usepackage{amsfonts}
\usepackage{amscd}
\usepackage{amssymb}
\usepackage{amsthm}
\usepackage{amstext}
\usepackage{enumerate}
\usepackage{latexsym,amssymb}

\usepackage{color}
\definecolor{micolor}{rgb}{0.85, 0.5, 0.4}
\newtheorem{theorem}{Theorem}[section]
\newtheorem{lemma}[theorem]{Lemma}

\newtheorem{proposition}[theorem]{Proposition}
\theoremstyle{definition}
\newtheorem{definition}[theorem]{Definition}

\theoremstyle{remark}

\numberwithin{equation}{section} \allowdisplaybreaks
\begin{document}

\title{Singular integrals with variable kernels in dyadic settings}
\vskip 0.3 truecm

\author{Hugo Aimar, Raquel Crescimbeni and Luis Nowak$\,$\thanks{This research is partially supported by Consejo Nacional de Investigaciones Cient\'ificas y T\'ecnicas, Universidad Nacional del Litoral and Universidad Nacional del Comahue, Argentina.\newline \indent Keywords and phrases:
Singular integrals, Spaces of homogeneous type, Petermichl's operator, Haar basis.
\newline \indent
\newline }}
\date{\vspace{-0.5cm}}
\maketitle


\date{}


\begin{abstract}
In this paper we explore conditions on variable symbols with respect to Haar systems, defining Calder\'on-Zygmund type operators with respect to the dyadic metrics associated to the Haar bases.We show that Petermichl's dyadic kernel can be seen as a variable kernel singular integral and we extend it to dyadic systems built on spaces of homogeneous type.
\end{abstract}

\section{Introduction}

The seminal work of A.P. Calder\'on and A. Zygmund during the fifties of the last century, regarding singular integrals and their relation to partial differential equations, can be considered the corner stone  of modern Harmonic Analysis, see E. Stein in \cite{St} for historical development of the ideas and their impact in the actual and future research in the area.
Let us point out two aspects of their contributions that will help us at introducing the problems that we consider in this paper. These aspects are contained in the two papers \cite{CZ1952} and \cite{CZ1978}. In \cite{CZ1952} the authors consider convolution type singular integral operators and in \cite{CZ1978} they introduce non-convolution type kernels, also called variable kernels. 

In the Calder\'on-Zygmund singular integral theory in metric and quasi-metric spaces (see \cite{CW}, \cite{MS1},\cite{MS2}, \cite{A} and \cite{DS}), the distinction between convolution and non-convolution kernels does not a priori make sense because convolution is not generally  defined in this setting.
Nevertheless, there is still another way to consider a convolution operator. The idea goes back to the works of  Mikhlin, Giraud and Tricomi (see \cite{M1}, \cite{M2} and the references therein) which, aside from the depth of the analytic tools, it  becomes relevant at generating convolution type filters in machine learning when the analysis is considered on non euclidean data.This way is provided by the spectral analysis of the operators, when it is available. Let us briefly sketch the basic idea in a general framework. Assume that $\{\varphi_k\}$ is an orthonormal basis for the space $L^2(X,\mu)$, where $X$ is a measure space and $\mu$ is a Borel measure. In analogy with the Fourier case we consider convolution type operators, bounded in $L^2(X,\mu)$, as a multiplier operators of the form $$T_\eta f(x)\ =\ \sum_k \eta_k <f,\varphi_k> \varphi_k(x),$$
with $\eta=\{\eta_k\}$ a bounded scalar sequence. Here $<f,g>$ denotes the usual scalar product in $L^2(X,\mu)$. On the other hand, if instead of a sequence $\{\eta_k\}$ we consider in the definition of $T$ a sequence of bounded functions of $x$, $\{\eta_k(x)\}$, i.e. 
$$T f(x)\ =\ \sum_k \eta_k(x) <f,\varphi_k> \varphi_k(x),$$ 
we  say that $T$ is an operator with variable kernel  given at least formally by 
$$K(x,y) \ =\ \sum_k \eta_k(x) \varphi_k(x) \varphi_k(y).$$

In the analysis of unconditionality  wavelet bases in functional Banach spaces, as $L^p(\mathbb{R}^n)$, the operator defined by $\displaystyle T_\eta f(x) = \int_{\mathbb{R}^n}K_\eta (x,y) f(y) dy$ with a kernel given by $K_\eta (x,y) = \sum_{h \in \mathcal{H}} \eta(h) h(x) h(y)$ where $\mathcal{H}$ is the classical Haar system in $\mathbb{R}^n$ and $\eta$ is some bounded sequence defined on $\mathcal{H}$, is a singular integral operator when we give to $\mathbb{R}^n$ a suitable metric structure (see \cite{ACGN}).
Since $K_\eta$ is not translation invariant, the operator $T_\eta$ is not a convolution type operator in the classical euclidean sense. Nevertheless, the spectral form of $K_\eta(x,y)$ given by its symbol $\eta: \mathcal{H} \longrightarrow \mathbb{R}$, with respect to the Haar basis $\mathcal{H}$  which is independent of the points $x$ and $y$, is a good reason to consider $K_\eta$ as a standard convolution type kernel. 

On the other hand, a kernel whose spectral Haar analysis takes the form $$K_\eta (x,y) = \sum_{h \in \mathcal{H}} \eta(h,x) h(x) h(y)$$
for some $\eta : \mathcal{H} \times \mathbb{R}^n  \longrightarrow \mathbb{R}$, can be considered a variable kernel. A special case of variable kernel $K_\eta$ is considered by S. Petermichl in \cite{Pe1} as we shall see in Section 2.

In this work we aim to explore conditions on the variable symbol $\eta(h,x)$ in order to get kernels defining Calder\'on-Zygmund type operators with respect to a suitable dyadic metric. The  construction of dyadic cubes due to M. Christ (see \cite{C}) in spaces of homogeneous type becomes a basic tool in order to consider the problem in these general settings. 

The paper is organized as follows. In Section 2 we consider the variable kernel structure of Petermichl's operator in $\mathbb{R}$. In Section 3 we  introduce the basic properties of spaces of homogeneous type and we  define the dyadic family $\mathcal{D}$, the Haar system $\mathcal{H}$ and the dyadic metric $\delta$ in this general setting. Section 4 is devoted to introduce and prove the main result of this work providing sufficient conditions in the multiplier sequence in order to obtain obtain a Calder\'on-Zygmund  operator. Finally, in Section 5 we build Petermichl type operators on spaces of homogeneous type.\\

Throughout this work, we denote by $C$ a
constant that may change from one occurrence to other.

\section{On the Calder\'on-Zygmund structure of Petermichl's kernel}

In \cite{Pe1}, S. Petermichl introduce a dyadic kernel given in terms of the Haar functions by
 $$P(x,y)\ =\  \sum_{I \in \mathcal{D}} h_I(y) [h_{I^-}(x) - h_{I^+}(x)]$$ for $x,y \in \mathbb{R}^+$, $\mathcal{D}$ the dyadic intervals in $\mathbb{R}^+$, $h_I$ the Haar wavelets with support in the dyadic interval $I$ and $h_{I^-}$, $h_{I^+}$ the Haar wavelets in the left and right halves of the dyadic interval $I$. The corresponding operator is given by $$\mathcal{P} f(x) \ = \ \sum_{I \in \mathcal{D}}<f,h_I> \left(h_{I^-}(x) - h_{I^+}(x)\right).$$
\noindent
This operator is used in \cite{Pe1} to provide an outstanding formula for the Hilbert transform.

In \cite{AG} the authors proved that the kernel $P(x,y)$ has a standard Calder\'on-Zygmund structure when we consider the theory of singular integrals extended to metric measure spaces or, more precisely, to spaces of homogeneous type (see definition in Section 3). In other words, they show that 
$$P(x,y) = \frac{\Omega(x,y)}{\delta(x,y)}$$
with $\delta(x,y) = |I(x,y)|$ where $I(x,y)$ is the smallest dyadic interval in $\mathbb{R}^+$ containing  $x$ and $y$. They also prove that $\Omega$ is bounded and smooth with respect to the ultrametric $\delta$.
Before moving to the abstract setting in order to extend $P$ and $\mathcal{P}$, in this section we prove two elementary properties of the Petermichl's kernel that we shall explore later in the general frame work. Set $\mathcal{H}$ and $\mathcal{D}$ to denote the Haar system and dyadic family respectively in $\mathbb{R}^+$. For $h\in \mathcal{H}$ we denote with $I(h)$ the interval support of $h$, and we consider as $I^{--}_h$ the left quarter of $I(h)$,  $I^{-+}_h$ as the second quarter, $I^{+-}_h$ as the third quarter and $I^{++}_h$ as the last quarter of $I(h)$. 

\begin{proposition}
\item[\textit{(a)}]The operator $\mathcal{P}$ can be written as a variable kernel singular integral operator, in fact
$$\mathcal{P} f(x)\ =\ \frac{1}{\sqrt[]{ 2}}\sum_{h \in \mathcal{H}} \eta(x,h) <f,h> h(x)$$
with $\eta(x,h) = 1$ if $x \in I^{--}_h \cup I^{+-}_h$ and $m(x,h) = -1$ if $x \in I^{-+}_h \cup I^{++}_h$.
\item[\textit{(b)}]
If $\mathcal{P}^*$ denotes the adjoint of $\mathcal{P}$, then $\mathcal{P}\mathcal{P}^* = \mathcal{P}^*\mathcal{P} = 2 \mathcal{I}$, twice the identity in $L^2(\mathbb{R}^+)$.

\end{proposition}

\begin{proof}
Let us start by proving {\textit{(a)}}. If we denote  with $h^-$ and $h^+$ the Haar wavelets in the left and right halves of the support of $h$, respectively, we have that the supports of $h(y)h(x)$ and $h(y)[h^-(x) - h^+(x)]$ coincide as subsets of $(\mathbb{R}^+)^2$. Then in the support of $h(x)h(y)$ we have that
\begin{eqnarray}
h(y)[h^-(x) - h^+(x)] & = & h(y)\frac{[h^-(x) - h^+(x)]}{h(y)h(x)}h(y)h(x) \nonumber \\
& = & \frac{1}{\sqrt[]{ 2}} \eta(x,h) h(y)h(x), \nonumber
\end{eqnarray}
as desired.\\
In order to prove {\textit{(b)}} observe that 
$$\mathcal{P}^* f(y) \ =\ \sum_{I \in \mathcal{D}} \left(\left<f,h_{I^-}\right> - \left<f,h_{I^+}\right>\right) h_I(y).$$
On the other hand, from the orthonormality of the system $\mathcal{H}$, for each $I \in \mathcal{D}$ we have that
$$\left<\sum_{J \in \mathcal{D}}<f,h_J> \left(h_{J^-} - h_{J^+}\right) \,,\,h_{I^-}\right> \ = \ <f,h_I><h_{I^-},h_{I^-}>$$
and
$$\left<\sum_{J \in \mathcal{D}}<f,h_J> \left(h_{J^-} - h_{J^+}\right)\,,\, h_{I^+}\right> \ = \ <f,h_I><h_{I^+},h_{I^+}>.$$
Therefore
\begin{eqnarray}
\mathcal{P}^*(\mathcal{P}f)(y) & = & \sum_{I \in \mathcal{D}} \left(<\mathcal{P}f,h_{I^-}> - <\mathcal{P}f,h_{I^+}>\right) h_I(y) \nonumber \\
& = & \sum_{I \in \mathcal{D}} \left<f,h_I\right> \left<h_{I^-},h_{I^-}\right> h_I(y) + \sum_{I \in \mathcal{D}} \left<f,h_I\right> \left<h_{I^+},h_{I^+}\right> h_I(y) \nonumber \\
& = & 2f, \nonumber
\end{eqnarray}
as desired.
\end{proof}

\section{Dyadic families and Haar systems in spaces of homogeneous type}

Let us first briefly recall the basic properties of the general theory of spaces of homogeneous type. Assume that $X$ is a set, a nonnegative symmetric function $d$ on
$X \times X$ is called a quasi-distance if there exists a constant
$K$ such that
\begin{equation*}
\ d(x,y) \leq {{K}} [d(x,z) + d(z,y)],
\end{equation*} for every $x,y,z \in
X$ , and $d(x,y) = 0$ if and only if $x=y$.

We shall say that $(X,d,\mu)$ is a space of homogeneous type if
$d$ is a quasi-distance on $X$, $\mu$ is a positive Borel measure
defined on a $\sigma$-algebra of subsets of $X$ which contains the
balls, and there exists a constant $A$ such that 
\begin{equation} \label{doubling}
0\ <\ \mu(B(x,2r))\ \leq \ A\ \mu(B(x,r)) \ < \ \infty
\end{equation}holds for every
$x\in X$ and every $r>0$. This property is usually named as the doubling condition.\\
The construction of dyadic type families of subsets in metric or quasi-metric spaces with some inner and outer metric control of the sizes of the dyadic sets is given in \cite{C}. These families satisfy all the relevant properties of the usual dyadic cubes in $\mathbb{R}^n$ and are the basic tool to build wavelets on a metric space of homogeneous type (see \cite{A} or \cite{ABI}). Actually Christ's construction in \cite{C} shows the existence of dyadic families in spaces of homogeneous type. Nevertheless, in order to define Haar wavelets all we need is a dyadic family satisfying the following properties that we state as a definition and we borrow from \cite{ABI}.
\begin{definition}\label{fliady} Let $(X,d,\mu)$ be a metric space of homogeneous type. We say that $\mathcal{D} = {\bigcup_{j \in \mathbb{Z}}}\mathcal{D}^j$ is a dyadic family on $X$ with parameter $\lambda \in (0,1)$ if each $\mathcal{D}^j$ is a family
of Borel subsets $Q$ of $X$, such that

\item[\textit{(d.1)}]\textit{for every $j \in \mathbb{Z}$ the cubes in
$\mathcal{D}^j$ are pairwise disjoint;}
\item[\textit{(d.2)}]\textit{for every $j \in \mathbb{Z}$ the family $\mathcal{D}^j$  covers $X$ in the sense that $X = \bigcup_{Q \in \mathcal{D}^j}Q;$}
\item[\textit{(d.3)}] \textit{if $Q \in \mathcal{D}^j$ and $i < j$, then there
exists a unique $\tilde{Q} \in \mathcal{D}^i$ such that $Q
\subseteq \tilde{Q}$;}
\item[\textit{(d.4)}] \textit{if $Q \in \mathcal{D}^j$ and  $\tilde{Q} \in
\mathcal{D}^i$ with $i \leq j$, then either $Q \subseteq
\tilde{Q}$ or $Q \cap \tilde{Q} = \emptyset$;}
\item[\textit{(d.5)}] \textit{there exist two constants $a_1$ and $a_2$
such that for each $Q \in \mathcal{D}^j$ there exists a point $x \in Q$ that satisfies $B(x,a_1 \lambda^j) \subseteq Q \subseteq B(x,a_2 \lambda^j)$.}

\end{definition}

The following properties can be deduced from \textit{(d.1)} to \textit{(d.5)}, see \cite{ABN}.

\begin{lemma}  Let $\mathcal{D}$ be a dyadic family, then

\item[\textit{(d.6)}] \textit{there exists a positive integer M depending on $a_i$, $i=1,2$ in $(d.5)$ and on the doubling constant $A$ in \eqref{doubling} such that for every $j \in \mathbb{Z}$ and all $Q\in  \mathcal{D}^j$ the inequalities $1 \leq \#(\mathcal{L}(Q)) \leq M$ hold, where $\mathcal{L}(Q) = \{Q^{'} \in \mathcal{D}^{j+1}: Q^{'} \subseteq Q\}$ and $\#(B)$ denote the cardinal of $B$;}
\item [\textit{(d.7)}] \textit{there exists a positive constant $C$ such that $\mu(Q) \leq C \mu(Q^{'})$ for all $Q \in \mathcal{\tilde{D}}$ and every $Q^{'} \in \mathcal{L}(Q)$.}

\end{lemma}

It is easy to give examples of dyadic systems $\mathcal{D}$ such that a dyadic cube $Q$ belong to different levels $j \in \mathbb{Z}$. Since we are interested in the identification of those scales and places of partition which shall give rise to the Haar functions, we consider the subfamily $\mathcal{\tilde{D}}$ of $\mathcal{D}$ given by 
$$\mathcal{\tilde{D}} = {\bigcup_{j \in \mathbb{Z}}}\mathcal{\tilde{D}}^{j},$$
with
$$\mathcal{\tilde{D}}^{j} = \{Q\in  \mathcal{D}^j :\#(\{Q^{'} \in \mathcal{D}^{j+1}: Q^{'} \subseteq Q\}) > 1\}.$$

Properties $\textit{(d.1)}$ to $\textit{(d.6)}$ allow us to obtain the following aditional properties for $\mathcal{\tilde{D}}$.

\begin{enumerate}
\item [\textit{(d.8)}]
The families $\mathcal{\tilde{D}}^{j}$, $j \in \mathbb{Z}$ are pairwise disjoints.
\item [\textit{(d.9)}]
The function $\mathcal{J}:\mathcal{\tilde{D}} \longrightarrow
\mathbb{Z}$ given by $Q \mapsto \mathcal{J}(Q)$ if $Q \in
\mathcal{\tilde{D}}^{\mathcal{J}(Q)}$ is well defined.
\end{enumerate}

Let $\mathcal{D}$ be a dyadic family. We define, for each dyadic cube $Q$ in $\mathcal{D}$, the quadrant of $X$ that contain the cube $Q$, {\bf{{C}}}$(Q)$, by
$${\textbf{{C}}}(Q) = \underset{\{Q^{'}\in \mathcal{D}: Q \subseteq Q^{'}\}}{\bigcup}Q^{'}.$$

Following the lines in \cite{ABI} for the case of Christ's dyadic cube, from {\textit{(d.6)}} and since all the dyadic cubes $Q$ in $\mathcal{D}$ are spaces of homogeneous type with uniform doubling  constant, we get that if $(X,d,\mu)$ is a space of homogeneous type and if $\mathcal{D}$ is a dyadic family, then there exists a positive integer $N$ (that depend of the geometric constants of $(X,d,\mu)$) and disjoint dyadic cubes $Q_{\alpha}, \alpha=1,...,N$ such that $$X  = \bigcup_{\alpha=1,...,N}{\textbf{{C}}}_{\alpha},$$
where ${\textbf{{C}}}_{\alpha} = {\textbf{{C}}}(Q_{\alpha})$. That is, there exists a finite number of quadrants these are a partition of $X$ and each one of them is a space of homogeneous type (see \cite{ABI})

In the classic euclidean context $\mathbb{R}^n$, the dyadic analysis leads to consider each quadrant separately. Then, without loss of generality, we will assume from now on that $X$ itself is a quadrant for $\mathcal{D}$.

Along this work,  given a dyadic family $\mathcal{D}$ we denote by $\delta(x,y)$ the dyadic metric associated to $\mathcal{D}$ for $x,y \in X$. That is $\delta$ is the function defined in $X \times X$ given by
\begin{align}\label{ultrametrica}
\delta (x,y)=
\begin{cases}
min\{\mu(Q): x,y \in Q, Q \in \tilde{\mathcal{D}}\}&\text{\,if\,} \ \  x \not= y \\
0 &\text{\,if\,} \ \ x=y.
\end{cases}
\end{align}

Now we state and prove the main result of this section. The proof follow the technique used in \cite{MS1} where the authors prove that each quasi-metric space $(X,d)$ is metrizable and that $d$ is equivalent to
$\rho^\beta$, where $\rho$ is a distance on $X$ and $\beta \geq 1$. Moreover, they show that all spaces of homogeneous type $(X,d,\mu)$ can be normalized in the  sense that there exists a metric $\rho$ on $X$ and two constants $C_1$ y $C_2$ such that 
\begin{equation}\label{normal}
C_1 r\leq \mu(B_{\rho}(x,r)) \leq C_2 r, 
\end{equation}
where $B_{\rho}(x,r)=\{y\in X \ : \ \rho(x,y)<r\}$.
In general, if $\rho$ satisfies (\ref{normal}), we say that $(X,\rho ,\mu)$ is a normal space of homogeneous type or $1-$Ahlfors.  
 
\begin{lemma}\label{normalizacion diadica}
Let $(X,d,\mu)$ be a space of homogeneous type and let $\mathcal{D}$ be a dyadic family. Then $(X,\delta,\mu)$ is a normal space of homogeneous type. Moreover, the characteristic functions of dyadic cubes are Lipschitz functions in $(X,\delta)$.
\end{lemma}

\begin{proof}

For each $z \in X$ we write $Q_j(z)$ to denote the unique dyadic cube $Q \in  \tilde{\mathcal{D}_j}$ such that $z \in Q$.
Without loss of generality we can assume that $X$ in not bounded.
Thus, if $x \in X$, $r>0$, and $j_0$ is an integer in $\mathbb{Z}$ such that 
\begin{equation}\label{elecion jcero}
\mu(Q_{j_0}(x)) \ \leq r \ < \ \mu(Q_{j_0 - 1}(x)),
\end{equation}
then 
\begin{equation}\label{bola igual cubo}
B_{\delta}(x,r) = Q_{j_0}(x).
\end{equation}
In fact if $y \in Q_{j_0}(x)$ then $x,y \in Q_{j_0}(x)$ and therefore $\delta(x,y) \leq \mu(Q_{j_0}(x)) \ \leq r$ this implies that $Q_{j_0}(x) \subseteq B_{\delta}(x,r)$.
On the other hand, let $y\in B_\delta(x,r)$, if $y \notin Q_{j_0}(x)$ then $Q_{j_0}(x) \cap Q_{j_0}(y) = \emptyset$. Let $n \in \mathbb{N}$ be the first positive integer such that $Q_{j_0}(y) \subseteq Q_{j_0-n}(x)$, then we get that $\delta(x,y) = \mu (Q_{j_0-n}(x)) \geq \mu(Q_{j_0-1}(x)) > r$, this is a contradiction. Hence $y \in Q_{j_0}(x)$ and then  $B_{\delta}(x,r) \subseteq Q_{j_0}(x)$.
In orden to prove that $(X,\delta,\mu)$ is a normal space of homogeneous type, observe that it is not difficult to see that $(X,\delta)$ is a metric space (see [2]) moreover, $\delta$ is an ultra-metric on $X$. Let $x \in X$ be and $r > 0$, consider the number $j_0$ given in (\ref{elecion jcero}). Since $B_{\delta}(x,r) = Q_{j_0}(x)$, we get that $\mu(B_{\delta}(x,r)) = \mu(Q_{j_0}(x) ) \leq r$.
On the other hand, since $Q_{j_0}(x) \subseteq \mathcal{L}(Q_{j_0-1}(x))$, by the doubling property of the measure \eqref{doubling} there exists a positive constant $C$ such that $\mu(Q_{j_0-1}(x)) \leq C \mu(Q_{j_0}(x))$,
then from (\ref{elecion jcero}) and (\ref{bola igual cubo}) we get that 
$$r < \mu(Q_{j_0-1}(x)) \leq C \mu(Q_{j_0}(x)) = C \mu(B_{\delta}(x,r)).$$
Hence, $\frac{r}{C} < \mu(B_{\delta}(x,r)).$
Finally, for the last statement, let $x,y \in X$ and $Q \in \tilde{\mathcal{D}}$. If $x,y \in Q$ or if $y \notin Q$, $x \notin Q$, then $\chi_{_Q}(x) - \chi_{_Q}(y) = 0$. If $Q$ contain only the point $x$ or the point $y$ and $Q(x,y)$ is the smallest dyadic cube such that $x,y \in Q(x,y)$, then $\delta(x,y) = \mu(Q(x,y)) \geq \mu(Q)$. Hence $|\chi_{_Q}(x) - \chi_{_Q}(y)| = 1 \leq \frac{1}{\mu(Q)}\delta(x,y)$.
\end{proof}
From now on we shall denote by $Q(x,y)$  the smallest dyadic cube such that $x,y \in Q(x,y)$.
From each dyadic system $\mathcal{D}$ as above we can associate a Haar type systems that we borrow from (\cite{ABN}).

\begin{definition}\label{haardyadic} Let $\mathcal{D}$ be a dyadic family on $(X,d,\mu)$. A system $\mathcal{H}$ of simple Borel measurable real functions $h$ on $X$ is said to be a Haar system associated to $\mathcal{D}$ if it is an orthonormal basis of $L^2(X,\mu)$ such that

\item[\textit{(h.1)}] \textit{ For each $h \in \mathcal{H}$ there exists a unique $j \in \mathbb{Z}$ and a cube
$Q(h) \in  \mathcal{\tilde{D}}^j$ such that $\{x \in X: h(x) \not=0\} \subseteq Q(h)$, and this property does not hold for any cube
in $\mathcal{D}^{j+1}$.}
\item[\textit{(h.2)}] \textit{ For every $Q \in \mathcal{\tilde{D}}$ there exist exactly $M_Q = \#(\mathcal{L}(Q)) - 1 \geq 1$ functions $h \in \mathcal{H}$ such
that (h.1) holds. We denote with $\mathcal{H}(Q)$  the set
of all these functions $h$.}
\item[\textit{(h.3)}] \textit{ For each  $h \in \mathcal{H}$ we have that $\int_X h
d \mu = 0$.}
\item[\textit{(h.4)}] \textit{ For each $Q\in  \mathcal{\tilde{D}}$ let $V_Q$ denote the
vector space of all functions on $Q$ which are constant on each
$Q^{'} \in \mathcal{L}(Q)$. Then the system
$\{\frac{\chi_{_{Q}}}{(\mu(Q))^{1/2}}\ \}\ \bigcup \mathcal{H}(Q)$
is an orthonormal basis for $V_Q$.}
\item[\textit{(h.5)}] \textit{ There exists a positive constant $C$ such that the inequality $|h(x)| \leq C |h(y)|$
holds for almost every $x$ and $y$ in $Q(h)$ and every $h \in \mathcal{H}$.}

\end{definition}

 Observe also that from \textit{(d.7)}, \textit{(h.4)} and \textit{(h.5)} we get that there exists two positive constants $C_1$ and $C_2$ such that
\begin{equation}\label{norma infinito de h}
\frac{C_1}{\mu(Q(h))^{1/2}} \leq |h(x)| \leq \frac{C_2}{\mu(Q(h))^{1/2}},
\end{equation}
for all $h \in \mathcal{H}$ and $x \in Q(h)$.

\section{On convolution and non-convolution type singular integral operators in metric measure spaces.}

Let $(X,d,\mu)$ a space of homogeneous type,  $\mathcal{D}$ and $\mathcal{H}$ the dyadic family of cubes and the Haar system associated given in Definitions  \ref{fliady} and \ref{haardyadic} respectively. For simplicity we denote by $L^2=L^2(X,\mu)$ of square integrable real functions defined on $X$. 
 Since $\mathcal{H}$ is an orthonormal basis for $L^2$, we have the resolution of the identity given by 

$$f \ = \ \sum_{h \in \mathcal{H}} \left<f,h\right> h.$$
The operators
\begin{equation}\label{op convolution}
T_\eta f(x) = \sum_{h \in \mathcal{H}} \eta(h)\left<f,h\right> h(x),
\end{equation}
with $\eta$ a  bounded function defined on $\mathcal {H}$,  or more generally
\begin{equation}\label{op no convolution}
T_\eta f(x) = \sum_{h \in \mathcal{H}} \eta(x,h)\left<f,h\right> h(x),
\end{equation}
with $\eta$ a bounded function defined on $X \times \mathcal{H} $, are bounded in $L^2$.

With the heuristics described in the introduction we may think that the operator  as in  (\ref{op convolution}) is of convolution type while that in (\ref{op no convolution}) is of non-convolution type singular. 
In this section we give a sufficient condition on $\eta(x,h)$ in such a way that $T_\eta$ defined by (4.2) becomes a Calder\'on-Zygmund type operator in $(X,d,\mu)$.

A bounded linear operator $T:L^2 \longrightarrow L^2$ is said to be of Calder\'on-Zygmund type in $(X,\delta ,\mu)$ if there exists $K \in L^1_{loc}(X \times X\setminus\Delta)$, with $\Delta$ the diagonal of $X \times X$, such that
\begin{enumerate}
\item [$(1)$]
 there exists a positive constant $C$ such that $|K(x,y)| \leq \frac{C}{\delta(x,y)}$ for $x,y \in X$ with $x \not= y$,
 
\item [$(2)$]
there exists two positive constants $C$ and $\gamma$ such that

\begin{enumerate}
\item[$(2.a)$] $|K(x',y) - K(x,y)| \leq C \frac{\delta(x',x)^\gamma}{\delta(x,y)^{1+\gamma}},$ \, \, \text{if} \, \, $2 \delta(x',x) \leq \delta(x,y)$;
\item[$(2.b)$]
$|K(x,y') - K(x,y)| \leq C \frac{\delta(y,y')^\gamma}{\delta(x,y)^{1+\gamma}},$ \, \, \text{if} \, \, $2 \delta(y',y) \leq \delta(x,y)$;
\end{enumerate}

\item [$(3)$]
for $\varphi , \psi \in \mathcal{S}(\mathcal{H})$, the linear span of $\mathcal{H}$, with $supp \varphi \cap supp \psi = \emptyset$, we have  
$$<T(\varphi), \psi>\ =\ \int\int_{X \times X}K(x,y) \varphi(x) \psi(y) d(\mu \times \mu)(x,y).$$
\end{enumerate}

The main result of this section is contained in the following statement.

\begin{theorem}\label{principal teo}
Let $(X,d,\mu)$ a space of homogeneous type,  $\mathcal{D}$ a dyadic family, $\mathcal{H}$ a Haar system and $\delta$ defined in \eqref{ultrametrica}. Let $\eta: X \times \mathcal{H} \longrightarrow \mathbb{R}$ be a function such that  is a measurable function in $x \in X$ for each $h\in \mathcal{H}$ and there exists a constant $B>0$ such that

\item[(a)]
$|\eta(x,h)| \leq B,$ for $x \in X$ and $h \in \mathcal{H}$
\item[(b)]
$|\eta(x',h) - \eta(x,h)| \leq B \frac{\delta(x,x')}{\mu(Q(h))}$, for $h \in \mathcal{H}$ and $x, x' \in X$. 

Then the operator 
$$T_\eta f(x) \ = \ \sum_{h \in \mathcal{H}} \eta(x,h)\left<f,h\right> h(x)$$
is of Calder\'on-Zygmund type in the space of homogeneous type $(X,\delta,\mu)$. Hence $T_\eta$ is bounded on $L^p(X)$ $(1 < p < \infty)$ and of weak type $(1,1)$.
\end{theorem}

\begin{proof}
The $L^2$ boundedness of $T_\eta$ follows from $(a)$ with $\|T_\eta f\|_2 \leq \|\eta\|_\infty \|f\|_2$. By testing $T_\eta$ with simple function in $\mathcal{S}(\mathcal{H})$, we see that 
$$K(x,y) \ =\ \sum_{h \in \mathcal{H}} \eta(x,h) h(y) h(x)$$
satisfies property $(3)$ in the above definition of Calder\'on-Zygmund kernel in the general setting. Let us prove $(1)$ of the definition of Calder\'on-Zygmund type operator. Let $x \not= y$ in $X$ and  $Q(x,y)$ in $\mathcal{D}$ such that $\mu(Q(x,y)) = \delta(x,y)$. On the other hand for any cube strictly smaller than $Q(x,y)$ we must have $h(y)=0$ or $h(x)=0$. Hence from $(h.1)$,  (\ref{norma infinito de h}) and $(d.6)$ we get 
\begin{eqnarray}
|K(x,y)| & \leq &  C\|\eta\|_\infty  \sum_{Q \supseteq Q(x,y)} \sum_{\{h \in \mathcal{H}: Q(h)=Q\}}\frac{1}{\mu(Q)}\nonumber \\
& \leq & C \|\eta\|_\infty M \sum_{Q \supseteq Q(x,y)}\frac{1}{\mu(Q)},\nonumber
\end{eqnarray}
where $M$ is as in $(d.6)$ in Lemma 3.2. Notice that we are considering only the cubes in $\mathcal{\tilde{D}}$. Then if $Q_m$ is the $m-$th ancestor of $Q(x,y)$ in $\mathcal{\tilde{D}}$, the measure of this sequence grows geometrically, i.e. 
\begin{equation}\label{duplicacion crec geometrico}
\mu(Q_m) \geq (1 + \varepsilon)^m \mu(Q(x,y))
\end{equation}
with a geometric constant $\varepsilon > 0$.
Hence 
$$|K(x,y)| \ \leq \ \frac{C  }{\delta(x,y)}$$ 
as desired.
Let us now prove the smoothness properties of $K$. Notice first that, from (\ref{duplicacion crec geometrico}) we get that
\begin{eqnarray}\label{suma geometrica Q}
\sum_{\underset{Q \supseteq Q(x,y)}{Q \in \tilde{\mathcal{D}}}}\ \frac{1}{(\mu(Q))^2} & = & \sum_{m \in \mathbb{N}}\ \frac{1}{(\mu(Q_{m-1}))^2}\nonumber \\
& \leq & \sum_{m \in \mathbb{N}}\ \left(\frac{1}{\left(1 + \varepsilon\right)^2}\right)^{m-1} \frac{1}{(\mu(Q_{0}))^2}\nonumber \\
& = & \frac{1}{(\mu(Q_{0}))^2} \sum_{m \in \mathbb{N}}\ \left(\frac{1}{\left(1 + \varepsilon\right)^2}\right)^{m-1}\nonumber \\
& = & \frac{C}{(\mu(Q_{0}))^2},
\end{eqnarray}
where $Q_0=Q(x,y) $ in $\mathcal{\tilde{D}}$.
On the other hand, notice that for $h \in \mathcal{H}$ if $Q=Q(h) \in  \tilde{\mathcal{D}}$, then 
$$h(x) \ = \ \sum_{Q' \in \mathcal{L}(Q)} \ \beta_{Q'} \chi_{_{Q'}}(x),$$
where $\beta_{Q'} \in \mathbb{R}$. Thus, since the characteristic functions on dyadic cube are Lipschitz functions on $(X,\delta)$, from dyadic doubling property, $ (d.6)$ and \eqref{norma infinito de h}  there exists a positive constant $C$ such that if $x, x' \in X$ we get that
\begin{eqnarray}\label{h en x menos xprima}
|h(x) - h(x')|
& \leq & \sum_{Q' \in \mathcal{L}(Q(h))}\ |\beta_{Q'}|\left|\chi_{_{Q'}}(x) - \chi_{_{Q'}}(x')\right|\nonumber \\
& \leq & \sum_{Q' \in \mathcal{L}(Q(h))}\ \|h\|_{\infty} \left|\chi_{_{Q'}}(x) - \chi_{_{Q'}}(x')\right|\nonumber \\
& \leq & \frac{C}{\sqrt[]{ \mu(Q(h))}}\sum_{Q' \in \mathcal{L}(Q(h))}\ \left|\chi_{_{Q'}}(x) - \chi_{_{Q'}}(x')\right|\nonumber \\
& \leq & \frac{C}{\sqrt[]{ \mu(Q(h))}}\sum_{Q' \in \mathcal{L}(Q(h))}\ \frac{\delta(x,x')}{\mu(Q')}\nonumber \\
& \leq & \frac{C^2}{\sqrt[]{ \mu(Q(h))}}\sum_{Q' \in \mathcal{L}(Q(h))}\ \frac{\delta(x,x')}{\mu(Q(h))}\nonumber \\
& \leq & \frac{C^2 \delta(x,x')}{(\mu(Q(h)))^{\frac{3}{2}}}\#\mathcal{L}(Q(h))\nonumber \\
& \leq & M C^2 \frac{ \delta(x,x')}{(\mu(Q(h)))^{\frac{3}{2}}} \nonumber \\
& \leq & C \frac{ \delta(x,x')}{\mu(Q(h))^{\frac{3}{2}}}.
\end{eqnarray}
Observe now that if $x, y, x' \in X$ satisfy  $2 \delta(x',x) \leq \delta(x,y)$ then $x' \in Q(x,y)$ and moreover
\begin{equation*}
Q(x,y) = Q(x',y).
\end{equation*}
In fact, if $x' \notin Q(x,y)$ then $\delta(x,x') > \delta(x,y)$, which is a contradiction. On the other hand, since $Q(x,y) \in \tilde{\mathcal{D}}$, there exists two different dyadic cubes $Q'$ and $\hat{Q}$ in $\mathcal{L}(Q(x,y))$ such that $y \in Q'$ and $x \in \hat{Q}$. So, if $x' \in X$ satisfies $2 \delta(x',x) \leq \delta(x,y)$ and we suppose that $x' \notin \hat{Q}$, then
$$\delta(x,x')\ =\ \mu(Q(x,y))\ =\ \delta(x,y),$$
which is again a contradiction.
Then  if $2 \delta(x',x) \leq \delta(x,y)$ we have $Q(x,y) = Q(x',y)$, this implies that $\delta(x,y) = \delta(x',y)$. Hence in such case, from the conditions $(a)$ and $(b)$ on $\eta$,  (\ref{h en x menos xprima}), (\ref{suma geometrica Q}) and \eqref{norma infinito de h} we get that
\begin{align*}
\left|\left(\eta(x',h) h(x') - \eta(x,h) h(x)\right) h(y)\right| = &
\left(\left|\eta(x',h)  - \eta(x,h) \right||h(x')| 
 +  |\eta(x,h)|\left| h(x')-  h(x)\right|\right)|h(y)|\nonumber \\
 \leq &\left(\frac{C B \delta(x,x')}{(\mu(Q(h)))^{3/2}} + B\frac{M C^2 \delta(x,x')}{(\mu(Q(h)))^{3/2}}\right)|h(y)|\nonumber \\
\leq& \left(\frac{C B \delta(x,x')}{(\mu(Q(h)))^{2}} + B\frac{M C^2 \delta(x,x')}{(\mu(Q(h)))^{2}}\right)\nonumber \\
 = & C\frac{ \delta(x,x')}{(\mu(Q(h)))^{2}}. \nonumber \\
 \end{align*}
 
Then from the above estimate and \eqref{suma geometrica Q} we get that

\begin{align}
|K(x',y) - K(x,y)|  &=  \left|\sum_{h \in \mathcal{H}}\left(\eta(x',h) h(x') - \eta(x,h) h(x)\right) h(y)\right|\nonumber \\
 &=   \left|\sum_{\underset{Q \supseteq Q(x,y)}{Q \in \tilde{\mathcal{D}}}}\ \ \sum_{\underset{Q(h) = Q}{h \in \mathcal{H}}}\left(\eta(x',h) h(x') - \eta(x,h) h(x)\right) h(y)\right|\nonumber \\
 &=  C\sum_{\underset{Q \supseteq Q(x,y)}{Q \in \tilde{\mathcal{D}}}}\ \ \sum_{\underset{Q(h) = Q}{h \in \mathcal{H}}}\frac{ \delta(x,x')}{(\mu(Q(h)))^{2}}\nonumber \\
 &\leq C \sum_{\underset{Q \supseteq Q(x,y)}{Q \in \tilde{\mathcal{D}}}}\ \frac{ \delta(x,x')}{(\mu(Q))^{2}}\nonumber \\
& =  
  C\frac{\delta(x,x')}{(\mu(Q(x,y)))^2}\nonumber \\
&=  C\frac{\delta(x,x')}{(\delta(x,y))^2},\nonumber
\end{align}
this complete the proof of $(2.a)$.
In a similar way we can prove $(2.b)$.

\end{proof}

\section{Petermichl's type operators in spaces of homogeneous type}

In this section we introduce Petermichl type operators $\mathcal{P}$ on spaces of homogeneous type. We prove, using Theorem \ref{principal teo}, that this operator is a Calder\'on-Zygmund type operator on a suitable  space of homogeneous type and we show that $\mathcal{P}^*$ is almost the identity operator in a sense that shall be made precise.

Let $(X,d,\mu)$ be a space of homogeneous type,  $\mathcal{D}$ a dyadic family, $\mathcal{H}$ a Haar system associated to $\mathcal{D}$ and $\left(\alpha_{h}\right)_{h \in \mathcal{H}}$  a bounded sequence in $\mathbb{R}$. For $f \in L^2(X,\mu)$ we consider the operator  $\mathcal{P}$ defined as

\begin{equation*}
\mathcal{P}f(x) = \sum_{Q \in \tilde{\mathcal{D}}}\ \ \sum_{\underset{Q(h) = Q}{h \in \mathcal{H}}} <f,h> \left(\sum_{\underset{R \in \mathcal{L}(Q)}{\tilde{h} \in \mathcal{H}(R)}} \alpha_{\tilde{h}}\tilde{h}(x) \right)
\end{equation*}
where we recall that $\mathcal{H}(R)$ is given in $(h.2)$.

\begin{proposition}\label{basicas prop de P}
Let $(X,d,\mu)$ be a space of homogeneous type,  $\mathcal{D}$ the dyadic family,  $\mathcal{H}$ the Haar system associated to $\mathcal{D}$ and  $\left(\alpha_{h}\right)_{h \in \mathcal{H}}$  a bounded sequence in $\mathbb{R}$. Then the operator $\mathcal{P}$  satisfies the following properties

\item [(1)]

$$\mathcal{P}f(x) = \int_{y \in X} N(x,y) f(y) d\mu (y),$$
where
$\displaystyle N(x,y) = \sum_{Q \in \tilde{\mathcal{D}}}\ \ \sum_{\underset{Q(h) = Q}{h \in \mathcal{H}}} h(y) \left(\sum_{\underset{R \in \mathcal{L}(Q)}{\tilde{h} \in \mathcal{H}(R)}} \alpha_{\tilde{h}}\tilde{h}(x) \right)$ 
and  $f$ is a simple function in $\mathcal{S}(\mathcal{H})$.
\item [(2)]

$$\mathcal{P}^*f(x) = \int_{y \in X} N^*(x,y) f(y) d\mu (y),$$
where $N^*(z,w) = N(w,z)$ and $f$ is a simple function in $\mathcal{S}(\mathcal{H})$.
\item [(3)]
$$\mathcal{P}^*(\mathcal{P}f)(x)=\sum_{h \in \mathcal{H}} C(Q) <f,h> h(x),$$ with  $1 \leq C(Q)\leq  M^2$, with $M$  as  in  $(d.6)$ in  Lemma~3.2.
\end{proposition}

\begin{proof}

In order to prove $(1)$, we observe that for $f$ in $\mathcal{S}(\mathcal{H})$ the sum in the definition of $\mathcal{P}f(x)$ is finite and therefore we have that
\begin{eqnarray}
\mathcal{P}f(x) & = & 
\int_{y \in X} \left(\sum_{Q \in \tilde{\mathcal{D}}}\ \ \sum_{\underset{Q(h) = Q}{h \in \mathcal{H}}} h(y) \left(\sum_{\underset{R \in \mathcal{L}(Q)}{\tilde{h} \in \mathcal{H}(R)}} \alpha_{\tilde{h}}\tilde{h}(x) \right)\right) f(y) d\mu (y) \nonumber \\
& = & \int_{y \in X} N(x,y) f(y) d\mu (y), \nonumber
\end{eqnarray}
where
$$N(x,y) = \sum_{Q \in \tilde{\mathcal{D}}}\ \ \sum_{\underset{Q(h) = Q}{h \in \mathcal{H}}} h(y) \left(\sum_{\underset{R \in \mathcal{L}(Q)}{\tilde{h} \in \mathcal{H}(R)}} \alpha_{\tilde{h}}\tilde{h}(x) \right).$$ 
On the other hand, $$\mathcal{P}^*f(z) = \int_{w \in X} N^*(z,w) f(w) d\mu (w),$$
for $N^*(z,w) = N(w,z)$.\\
Finally we compute the action of $\mathcal{P}^*$ on $\mathcal{P}$. By Fubini's theorem we get that

\begin{eqnarray}
\mathcal{P}^*(\mathcal{P}f)(x) & = & \int_{y \in X} N^*(x,y) \mathcal{P}f(y) d\mu (y)\nonumber \\
& = & \int_{y \in X} N^*(x,y) \int_{z \in X} N(y,z) f(z) d\mu (z) d\mu (y)\nonumber \\
& = & \int_{y \in X} N(y,x) \int_{z \in X} N(y,z) f(z) d\mu (z) d\mu (y)\nonumber \\
& = & \int_{z \in X} \left(\int_{y \in X} N(y,x)N(y,z) d\mu (y) \right) f(z) d\mu (z) \nonumber \\
& = & \int_{z \in X} U(x,z) f(z) d\mu (z), \nonumber 
\end{eqnarray}
where
\begin{eqnarray}
U(x,z) & = & \int_{y \in X} N(y,x)N(y,z) d\mu (y)\nonumber \\
& = &  \sum_{Q \in \tilde{\mathcal{D}}} \sum_{\underset{Q(h) = Q}{h \in \mathcal{H}}} \sum_{Q' \in \tilde{\mathcal{D}}} \sum_{\underset{Q(h') = Q}{h' \in \mathcal{H}}} h(x) h'(z) \int_{y \in X} \sum_{\underset{R \in \mathcal{L}(Q)}{\tilde{h} \in \mathcal{H}(R)}} \alpha_{\tilde{h}}\tilde{h}(y) \sum_{\underset{R' \in \mathcal{L}(Q')}{\hat{h} \in \mathcal{H}(R')}} \alpha_{\hat{h}}\hat{h}(y) d\mu (y).\nonumber
\end{eqnarray}
Now, by the orthogonality of the Haar system

\begin{eqnarray}
\int_{y \in X} \sum_{\underset{R \in \mathcal{L}(Q)}{\tilde{h} \in \mathcal{H}(R)}} \alpha_{\tilde{h}}\tilde{h}(y) \sum_{\underset{R' \in \mathcal{L}(Q')}{\hat{h} \in \mathcal{H}(R')}} \alpha_{\hat{h}}\hat{h}(y) d\mu (y) & = &  \sum_{\underset{R \in \mathcal{L}(Q)}{\tilde{h} \in \mathcal{H}(R)}}  \sum_{\underset{R' \in \mathcal{L}(Q')}{\hat{h} \in \mathcal{H}(R')}} \alpha_{\tilde{h}} \alpha_{\hat{h}} \int_{y \in X}\tilde{h}(y) \hat{h}(y) d\mu (y)\nonumber \\
& = & \sum_{\underset{R \in \mathcal{L}(Q)}{\tilde{h} \in \mathcal{H}(R)}} \alpha_{\tilde{h}}^2 \nonumber \\
& = & \left( \#(\mathcal{L}(Q)\right))\left(\#(\mathcal{L}(R) - 1)\right)\nonumber\\
&=& C_{(Q)} \nonumber .
\end{eqnarray}
Therefore

\begin{eqnarray}
U(x,z) & = & \int_{y \in X} N(y,x)N(y,z) d\mu (y)\nonumber \\
& = & \sum_{Q \in \tilde{\mathcal{D}}} \sum_{\underset{Q(h) = Q}{h \in \mathcal{H}}}  C_{(Q)}
h(x) h(z) .\nonumber
\end{eqnarray}
Thus

\begin{equation*}
\mathcal{P}^*(\mathcal{P}f)(x) = \int_{z \in X}\left(\sum_{Q \in \tilde{\mathcal{D}}} \sum_{\underset{Q(h) = Q}{h \in \mathcal{H}}} C_{(Q)} h(x)h(z) \right)f(z)d\mu (z).
\end{equation*}
with
\begin{equation*}
1\ \leq C_{(Q)}=\ \left(\#(\mathcal{L}(Q)\right)\left(\#(\mathcal{L}(R) - 1\right)\  \leq \ M^2
\end{equation*}  as desired.

\end{proof}

As an application of Theorem \ref{principal teo} we obtain the boundedness of these operators in Lebesgue spaces.

\begin{theorem}
Let $(X,d,\mu)$ be a space of homogeneous type. Let $\mathcal{D}$, $\mathcal{H}$ and $\delta$ be a dyadic family, a Haar systems associated to $\mathcal{D}$ and the dyadic metric induced by $\mathcal{D}$ respectively. Let $\left(\alpha_{h}\right)_{h \in \mathcal{H}}$  be a bounded sequence in $\mathbb{R}$. Then the operator 
$$\mathcal{P}f(x) = \sum_{Q \in \tilde{\mathcal{D}}}\ \ \sum_{\underset{Q(h) = Q}{h \in \mathcal{H}}} <f,h> \left(\sum_{\underset{R \in \mathcal{L}(Q)}{\tilde{h} \in \mathcal{H}(R)}} \alpha_{\tilde{h}}\tilde{h}(x) \right)$$
is a Calder\'on-Zygmund type operator on the space $(X,\delta,\mu)$. Hence $\mathcal{P}$ is bounded in $L^p(X)$ $(1 < p < \infty)$ and of weak type $(1,1)$.
\end{theorem}
 
\begin{proof}
By Theorem \ref{principal teo} it is enough  to prove that the operator $\mathcal{P}$ 
can be written as
$$\mathcal{P}f(x)\ = \ \sum_{h \in \mathcal{H}} \eta(x,h)\left<f,h\right> h(x)$$
for some function $\eta: X \times \mathcal{H} \longrightarrow \mathbb{R}$ 
satisfying the hypothesis in Theorem \ref{principal teo}. In fact for $h \in \mathcal{H}$ with $Q = Q(h) \in \mathcal{D}$ we have that 
$$h(x) \ = \ \sum_{R \in \mathcal{L}(Q(h))} h_{_R} \ \chi_{_R}(x),$$
where $h_{R} \in \mathbb{R}$. 
Thus, as $h$ is different from zero on $Q(h)$, we define for $x \in X$, 
$$\eta(x,h) \ = \ \sum_{R \in \mathcal{L}(Q(h))} \left(\sum_{\tilde{h} \in \mathcal{H}(R)} \frac{\alpha_{\tilde{h}}}{h_{_R}}\tilde{h}(x)\right)\chi_{_R}(x),$$
which is a measurable function for $x \in X$.
Then we get that
$$\eta(x,h) h(x) \ = \ \sum_{\underset{R \in \mathcal{L}(Q)}{\tilde{h} \in \mathcal{H}(R)}} \alpha_{\tilde{h}}\tilde{h}(x)$$
and therefore
$$\mathcal{P}f(x)\ = \ \sum_{h \in \mathcal{H}} \eta(x,h)\left<f,h\right> h(x).$$
Let us first prove that the function $\eta$ satisfies condition $(a)$ in the Theorem \ref{principal teo}.  Notice that if $h \in \mathcal{H}$ and $x \notin Q(h)$ then $\eta(x,h) = 0.$ On the other hand if $x \in Q(h)$, from \eqref{norma infinito de h},  doubling property on dyadic cubes $(d.7)$,  $(d.6)$ and  $(h.2)$  we get 
\begin{eqnarray}\label{prueba de a para nuestra eta}
|\eta(x,h)| 
& \leq & \sum_{R \in \mathcal{L}(Q(h))} \sum_{\tilde{h} \in \mathcal{H}(R)} \frac{|\alpha_{\tilde{h}}|}{|h_{_R}|}|\tilde{h}(x)|\left|\chi_{_R}(x)\right|\nonumber \\
& \leq & \sum_{R \in \mathcal{L}(Q(h))} \sum_{\tilde{h} \in \mathcal{H}(R)} \|(\alpha_{\tilde{h}})\|_\infty \frac{\sqrt[]{ \mu(Q(h))}}{C_1} \frac{C_2}{\sqrt[]{ \mu(Q(\tilde{h}))}}\nonumber \\
& \leq & \|(\alpha_{\tilde{h}})\|_\infty \sqrt[]{ C} \frac{C_2}{C_1} \left(\sum_{R \in \mathcal{L}(Q(h))} \sum_{\tilde{h} \in \mathcal{H}(R)} 1\right) \nonumber \\
& \leq & M^2 \|(\alpha_{\tilde{h}})\|_\infty \sqrt[]{ C} \frac{C_2}{C_1}=B,
\end{eqnarray}
where $M$ is as in $(d.6)$ in Lemma 3.2.

\noindent
In order to prove  that the function $\eta$ satisfies $(b)$ in Theorem 4.1, take $h \in \mathcal{H}$ with $Q = Q(h) \in \mathcal{D}$ as above  
$h(x) \ = \ \sum_{R \in \mathcal{L}(Q(h))} h_{_R} \ \chi_{_R}(x).$
We split the proof in five cases.

{\bf{Case 1}.} $x,x' \notin Q(h)$. Then $|\eta(x,h) - \eta(x',h)| = 0.$

{\bf{Case 2}.} $x,x' \in Q'$ for some $Q' \in \mathcal{L}(R_0)$ and some $R_0 \in \mathcal{L}(Q(h))$. Then, since in such case $\tilde{h}(x) = \tilde{h}(x')$ for every $\tilde{h} \in \mathcal{H}(R_0)$, we have that
\begin{eqnarray}
|\eta(x,h) - \eta(x',h)| & = & \left|\sum_{R \in \mathcal{L}(Q(h))} \left[\left(\sum_{\tilde{h} \in \mathcal{H}(R)} \frac{\alpha_{\tilde{h}}}{h_{R}}\tilde{h}(x)\right)\chi_{_R}(x) \ - \  \left(\sum_{\tilde{h} \in \mathcal{H}(R)} \frac{\alpha_{\tilde{h}}}{h_{R}}\tilde{h}(x')\right)\chi_{_R}(x')\right]\right|\nonumber \\
& = & \left| \sum_{\tilde{h} \in \mathcal{H}(R_0)} \left(\frac{\alpha_{\tilde{h}}}{h_{_{R_0}}}\tilde{h}(x) - \frac{\alpha_{\tilde{h}}}{h_{_{R_0}}}\tilde{h}(x')\right)\chi_{_{R_0}}(x)\right|=0\nonumber
\end{eqnarray}

{\bf{Case 3}.} $x \in Q$ and $x' \in Q'$ with $Q,Q' \in \mathcal{L}(R_0)$ and $R_0 \in \mathcal{L}(Q(h))$. Then, from (\ref{h en x menos xprima}),  (\ref{norma infinito de h}), doubling property on dyadic cubes $(d.7)$, $(d.6)$ and $(h.2)$ we get that 
\begin{eqnarray}
|\eta(x,h) - \eta(x',h)| & = & \left|\sum_{R \in \mathcal{L}(Q(h))} \left[\left(\sum_{\tilde{h} \in \mathcal{H}(R)} \frac{\alpha_{\tilde{h}}}{h_{_R}}\tilde{h}(x)\right)\chi_{_R}(x) \ - \  \left(\sum_{\tilde{h} \in \mathcal{H}(R)} \frac{\alpha_{\tilde{h}}}{h_{R}}\tilde{h}(x')\right)\chi_{_R}(x')\right]\right|\nonumber \\
& = & \left|\left(\sum_{\tilde{h} \in \mathcal{H}(R_0)} \frac{\alpha_{\tilde{h}}}{h_{_{R_0}}}\tilde{h}(x)\ -\ \sum_{\tilde{h} \in \mathcal{H}(R_0)} \frac{\alpha_{\tilde{h}}}{h_{_{R_{0}}}}\tilde{h}(x')\right)\chi_{_{R_{0}}}(x)\right|\nonumber \\
& = & \left|\sum_{\tilde{h} \in \mathcal{H}(R_0)} \frac{\alpha_{\tilde{h}}}{h_{R_{0}}}\left(\tilde{h}(x)\ - \tilde{h}(x')\right)\right|\nonumber \\
& \leq & \|(\alpha_{\tilde{h}})\|_\infty \frac{1}{|h_{R_{0}}|} M C^2 \delta(x,x') \sum_{\tilde{h} \in \mathcal{H}(R_0)}\frac{1}{(\mu(Q(\tilde{h})))^{3/2}}\nonumber \\
& \leq & \|(\alpha_{\tilde{h}})\|_\infty \frac{(\mu(Q(h)))^{1/2}}{C_1} M C^2 \delta(x,x') \sum_{\tilde{h} \in \mathcal{H}(R_0)}\frac{(\mu(Q(h)))^{3/2}}{(\mu(Q(h)))^{3/2}(\mu(Q(\tilde{h})))^{3/2}}\nonumber \\
& \leq & \|(\alpha_{\tilde{h}})\|_\infty \frac{M C^{5/2}}{C_1}  \frac{\delta(x,x')}{\mu(Q(h))} \left(\sum_{\tilde{h} \in \mathcal{H}(R_0)} 1\right) \nonumber \\
& \leq & \|(\alpha_{\tilde{h}})\|_\infty \frac{M^2 C^{5/2}}{C_1}  \frac{\delta(x,x')}{\mu(Q(h))}.\nonumber
\end{eqnarray}

{\bf{Case 4}.} $x \in Q(h)$ and $x' \notin Q(h)$ then $\eta(x',h) = 0$, also $\delta(x,x') > \mu(Q(h))$.  Hence, from (\ref{prueba de a para nuestra eta}) we obtain that
\begin{eqnarray}
|\eta(x,h) - \eta(x',h)| & = & |\eta(x,h)|\nonumber \\
& \leq & B  \nonumber \\
& \leq & B \frac{\delta(x,x')}{\mu(Q(h))}.\nonumber
\end{eqnarray}

{\bf{Case 5}.} $x \in R_1$ and $x' \in R_2$ with $R_1, R_2 \in \mathcal{L}(Q(h))$different. Then $\delta(x,x') = \mu(Q(h))$ and hence from (\ref{prueba de a para nuestra eta}) we get that
\begin{eqnarray}
|\eta(x,h) - \eta(x',h)| & \leq & |\eta(x,h)| + |\eta(x',h)|\nonumber \\
& \leq & 2  B \nonumber \\
& = & 2  B \frac{\delta(x,x')}{\mu(Q(h))}.\nonumber
\end{eqnarray}
as desired.
\end{proof}


\end{document}